\documentclass[11pt, reqno]{amsart}

\usepackage{amsmath,amsfonts,amssymb,amsthm,enumerate,hyperref,multicol}
\newtheorem{lemma}{Lemma}
\newtheorem{theorem}{Theorem}

\newtheorem{remark}{Remark}

\begin{document}
\leftline{ \scriptsize \it }

\title[Moment Estimates]{A note on Jain modified Phillips operators}
\maketitle

\begin{center}
{\bf Vijay Gupta} \\
Department of Mathematics, \\
Netaji Subhas Institute of Technology \\
Sector 3 Dwarka, New Delhi-110078, India \\
vijaygupta2001@hotmail.com \\
\vskip0.2in
{\bf G. C. Greubel} \\
Newport News, VA, United States \\
jthomae@gmail.com
\end{center}

\vspace{3mm}

\noindent \textbf{Abstract.} In the present article we define the Phillips type modification of the generalized
Sz{\'a}sz-Mirakjan operators. Moments, recurrence formulas, and other identities are established for these operators.
Approximation properties are also obtained with use of the Boham-Korovkin theorem.



\section{Introduction}
To obtain the generalization of the well known Sz\'asz-Mirakjan operators, Jain, \cite{J}, introduced the
following operators
\begin{align}\label{e1}  
B_{n}^{\beta}(f,x) = \sum_{k=0}^\infty L_{n,k}^{(\beta)}(x) \, f(k/n), \hspace{10mm} x \in [0,\infty),
\end{align}
where $0 \leq \beta<1$, and the basis functions are defined as
\begin{align*}
L_{n,k}^{(\beta)}(x) = \frac{nx(nx+k\beta)^{k-1}}{k!} \, e^{-(nx+k\beta)}
\end{align*}
with the normalization $\sum_{k=0}^{\infty}L_{n,k}^{(\beta)}(x)=1$. As a special case when $\beta=0$,
these operators, (\ref{e1}), reduce to the well-known Sz\'{a}sz-Mirakyan operators. It can be observed that
the extension of Jain's operators \cite{J} may be utilized to the generalization of Phillips type operators,
(\ref{e2}), in approximation theory. The main advantage is that such a generalization is possible and, for
large $n$, the convergence is faster. Recently Gupta-Malik \cite{vgnm} and Dhamiha-Deo \cite{nd} took the above
weights of the Jain basis function to define the modifications of generalized Baskakov operators, also Gupta and 
Agarwal, in \cite{vgrp}, compiled the results concerning convergence behavior of different operators for 
which this convergence is justified.

Recently Gupta and Greubel, \cite{vggr}, proposed the Durrmeyer type modification of the operators given in
(\ref{e1}) and established some direct results. Here we propose the generalized Phillips type operators for
$x\in [0,\infty)$ as
\begin{align*}
P_{n}^{\beta}(f,x) &= \sum_{k=1}^{\infty}\left(\int_0^\infty L_{n,k-1}^{(\beta)}(t)\,dt \right)^{-1}
L_{n,k}^{(\beta)}(x)\int_{0}^{\infty}L_{n,k-1}^{(\beta)}(t)f(t)\,dt+e^{-nx}f(0) \nonumber\\
\end{align*}
\begin{align}\label{e2} 
&= \sum_{k=1}^{\infty} \frac{< L_{n,k-1}^{(\beta)}(t), f(t) >}{< L_{n,k-1}^{(\beta)}(t), 1 >} \, L_{n,k}^{(\beta)}
(x)+e^{-nx}f(0)
\end{align}
where $< f,g > = \int_{0}^{\infty} f(t) \, g(t)dt$. For the case of $\beta = 0$, these operators reduce to the
Phillips operators (see \cite{RSP}). We may observe here that unlike the Phillips operators, these operators 
preserve only the constant functions. In the present article we obtain moments, identities for the coefficients
and the operators, recurrence identities and some direct estimates.

It should be mentioned that the polynomial form presented in (\ref{e11}) suggests that a reduction,
or more compact form is obtainable. Lemma 5 introduces these new polynomials for which other identities can be
obtained without carrying the terms involving $1- e^{-n x}$. Form and function are not lost by doing so and
Lemma 6 takes advantage of this reduction.


\section{Moments}
\begin{lemma} \cite{J}, \cite{af} \label{l1}
For the operators defined by (\ref{e1}) the moments are as follows:
\begin{align}
B_{n}^{\beta}(1,x) &= 1, \hspace{10mm} B_n^{\beta}(t,x)=\frac{x}{1-\beta} \nonumber\\
B_{n}^{\beta}(t^{2},x) &= \frac{x^2}{(1-\beta)^2}+\frac{x}{n \, (1-\beta)^3}, \nonumber\\
B_{n}^{\beta}(t^{3},x) &= \frac{x^{3}}{(1-\beta)^{3}} + \frac{3 \, x^{2}}{n \, (1-\beta)^{4}} + \frac{(1+2\beta)
\, x}{n^2 \, (1-\beta)^5} \label{e3} \\  
B_{n}^{\beta}(t^{4},x) &= \frac{x^{4}}{(1-\beta)^{4}} + \frac{6 \, x^{3}}{n (1-\beta)^{5}} + \frac{(7+8\beta)
\, x^{2}}{n^2 \, (1-\beta)^{6}} + \frac{(1 + 8\beta + 6 \beta^2) \, x}{n^{3} \, (1-\beta)^{7}} \nonumber\\
B_{n}^{\beta}(t^{5},x) &= \frac{x^{5}}{(1-\beta)^{5}} + \frac{10 \, x^{4}}{n \, (1-\beta)^{6}} + \frac{5 \,
(5 + 4\beta) \, x^{3}}{n^{2} \, (1-\beta)^{7}} \nonumber\\
& \hspace{10mm} + \frac{15 \, (1 + 4\beta + 2\beta^2) \, x^{2}}{n^{3} \, (1-\beta)^{8}} + \frac{(1 + 22\beta
+ 58\beta^2 + 24\beta^3) \, x}{n^{4} \, (1-\beta)^{9}} \nonumber
\end{align}
\end{lemma}

\begin{lemma} \label{l2} For $0\le \beta<1$, we have
\begin{align}\label{e4}  
\frac{< L_{n,k-1}^{(\beta)}(t), t^{r} >}{< L_{n,k-1}^{(\beta)}(t), 1 >} = P_{r}(k-1; \beta)
\end{align}
where $<f,g> =\int_{0}^{\infty} f(t) \, g(t)dt$ and $P_{r}(k-1; \beta)$ is a polynomial of order $r$ in the
variable $k$.
In particular
\begin{align}
P_{0}(k-1; \beta) &= 1 \nonumber\\
P_{1}(k-1; \beta) &= \frac{1}{n} \left[ (1-\beta) \, k + \frac{\beta \, (2-\beta)}{1-\beta} \right], \nonumber\\
P_{2}(k-1; \beta) &= \frac{1}{n^{2}} \left[ (1-\beta)^{2} \, k^{2} + a_{1}^{2} \, k + \frac{\beta^2 \,
(3-\beta)}{1-\beta} \right],
\label{e5} \\  
P_{3}(k-1; \beta) &= \frac{1}{n^{3}} \left[(1-\beta)^{3} \, k^3 + 3 \, a_{1}^{3} \, k^{2}
+ \frac{a_{2}^{3} \, k}{1-\beta} + \frac{\beta^{3} \, (4-\beta)}{1-\beta} \right] \nonumber
\end{align}
\begin{align*}
P_{4}(k-1; \beta) &= \frac{1}{n^{4}} \left[ (1-\beta)^{4} \, k^{4} + 2 \, a_{1}^{4} \, k^{3} +
a_{2}^{4} \, k^{2} + \frac{2 \, a_{3}^{4} \, k}{1-\beta} + \frac{\beta^{4} \, (5-\beta)}{1-\beta} \right] \nonumber\\
P_{5}(k-1; \beta) &= \frac{1}{n^{5}} \left[ (1-\beta)^{5} \, k^{5} + 5 \, a_{1}^{5} \, k^{4} + 5 \,
a_{2}^{5} \, k^{3} + \frac{5 \, a_{3}^{5} \, k^{2}}{1-\beta} + \frac{a_{4}^{5} \, k}{1-\beta} +
\frac{\beta^{5} \, (6-\beta)}{1-\beta} \right] \nonumber
\end{align*}
where
\begin{align*}
a_{1}^{2} &= 1 + 4\beta - 2\beta^2 \\
a_{1}^{3} &= 1 + \beta - 3\beta^2 + \beta^3
\hspace{25mm} a_{2}^{3} = 2 + 4\beta + 6\beta^2 - 12\beta^3 + 3\beta^4 \\
a_{1}^{4} &= 3 - 2\beta -7\beta^2 + 8\beta^3 - 2\beta^4
\hspace{10mm} a_{2}^{4} = 11 + 16\beta + 6\beta^2 - 24\beta^3 + 6\beta^4 \\
& \hspace{10mm} a_{3}^{4} = 3 + 5\beta + 5\beta^2 + 5\beta^3 - 10\beta^4 + 2\beta^5 \\
a_{1}^{5} &= (1-\beta)^{3} \, (2 + 2\beta - \beta^2)
\hspace{17mm} a_{2}^{5} = (1-\beta) \, (7 + 8\beta - 8\beta^3 + 2\beta^4) \\
& \hspace{10mm} a_{3}^{5} = 10 + 6\beta - 3\beta^2 - 8\beta^3 -12\beta^4 + 12\beta^5 - 2\beta^6 \\
& \hspace{10mm} a_{4}^{5} = 24 + 36\beta + 30\beta^2 + 20\beta^3 + 15\beta^4 - 30\beta^5 + 5\beta^6.
\end{align*}
\end{lemma}

\begin{proof}
First, consider the integral:
\begin{align*}
<L_{n,k-1}^{(\beta)}(t), t^{r} > &= \int_{0}^{\infty} L_{n,k-1}^{(\beta)}(t) \, t^{r} \, dt  \\
&= \frac{n}{\Gamma(k)}\int_{0}^{\infty} e^{-(nt+(k-1)\beta)} \, t^{r+1} \, (nt+(k-1)\beta)^{k-2} \, dt.
\end{align*}
With use of Tricomi's confluent hypergeometric function:
\begin{align*}
U(a,b,z) = \frac{1}{\Gamma(a)} \, \int_{0}^{\infty} e^{-zt} \, t^{a-1} \, (1+t)^{b-a-1} \, dt, \hspace{5mm} a>0,z>0
\end{align*}
we have
\begin{align}
<L_{n,k-1}^{(\beta)}(t), t^{r} > &= \frac{n}{\Gamma(k)} \, \int_{0}^{\infty} e^{-(nt+(k-1)\beta)} \, t^{r+1}
\, (nt+(k-1)\beta)^{k-2} \, dt \nonumber\\
&= \frac{((k-1)\beta)^{k+r}}{\Gamma(k) \, n^{r+1}} \, e^{-(k-1)\beta} \, \int_{0}^{\infty} e^{-(k-1)\beta t} \,
(1+t)^{k-2} \, t^{r+1} \, dt \nonumber\\
&= \frac{\Gamma(r+2) \, ((k-1)\beta)^{k+r}}{\Gamma(k) \, n^{r+1}} \, e^{-(k-1)\beta} \, U(r+2,k+r+1,(k-1)\beta). \label{e6}
\end{align}
By using
\begin{align*}
U(a,b,z) = z^{-a} \, {}_{2}F_{0} \left( a, a-b+1; - ; - \frac{1}{z} \right)
\end{align*}
and
\begin{align*}
{}_{2}F_{0}(-n, a; - ; z) = (-1)^{n} \, (a)_{n} \, z^{n} \, {}_{1}F_{1}\left(-n; 1-a-n; - \frac{1}{x} \right)
\end{align*}
then
\begin{align}\label{e7}  
<L_{n,k-1}^{(\beta)}(t), t^{r} > = \frac{(k)_{r}}{n^{r+1}} \, e^{-(k-1) \, \beta} \, {}_{1}F_{1}(2-k; 1-r-k; (k-1)\beta).
\end{align}
Divide both sides of (\ref{e7}) by $<L_{n,k-1}^{(\beta)}(t), 1 >$ to obtain
\begin{align}\label{e8}  
P_{r}(k-1; \beta) = \frac{(k)_{r}}{n^{r}} \, \frac{{}_{1}F_{1}(2-k; 1-r-k; (k-1)\beta)}
{{}_{1}F_{1}(2-k; 1-k; (k-1)\beta)}.
\end{align}
It is fairly evident that $P_{r}(k-1;\beta)$ is a polynomial in $k$ of order $r$.
\end{proof}

\begin{lemma}\label{l3}
For $0 \leq \beta < 1$, $r \geq 0$, the polynomials $P_{r}(k; \beta)$ satisfy the recurrence relationship
\begin{align}\label{e9}  
n^{2} P_{r+2}(k-1; \beta) = n [ (1-\beta) (k-1) + r + 2 ] P_{r+1}(k-1; \beta) + \beta (r+2) (k-1) P_{r}(k-1;
\beta).
\end{align}
\end{lemma}

\begin{proof}
By utilizing the recurrence relation, \cite{Bateman},
\begin{align*}
b(b-1) \, {}_{1}F_{1}(a; b-1; z) - b(b-1+z) \, {}_{1}F_{1}(a; b; z) - (a-b) z \, {}_{1}F_{1}(a; b+1; z) = 0,
\end{align*}
for the confluent hypergeometric functions it is evident that
\begin{align*}
0 &= \beta (r+1)(k-1) \, {}_{1}F_{1}(2-k; 2-r-k; (k-1)\beta) + (r+k-1)[(1-\beta)(k-1) + r + 1] \nonumber\\
& \hspace{10mm} \cdot  {}_{1}F_{1}(2-k; 1-r-k; (k-1)\beta) + (r+k)(1-r-k) \, {}_{1}F_{1}(2-k;-r-k; (k-1)\beta).
\end{align*}
Now dividing by ${}_{1}F_{1}(2-k; 1-k; (k-1)\beta)$, with use of (\ref{e8}), leads to the desired relationship
for the polynomials $P_{r}(k-1; \beta)$ given by (\ref{e9}).
\end{proof}

\begin{lemma}\label{l4}
If the $r$-th order moment with monomials $e_{r}(t)=t^{r},r=0,1,\cdots$ of the operators (\ref{e2}) be defined as
\begin{align*}
T_{n,r}^{\beta}(x):P_{n}^{\beta}(e_{r},x) = \sum_{k=1}^{\infty}\left(\int_0^\infty L_{n,k-1}^{(\beta)}(t)\,dt \right)^{-1}
L_{n,k}^{(\beta)}(x) \, \int_{0}^{\infty} L_{n,k-1}^{(\beta)}(t) \, t^{r} \, dt
\end{align*}
or
\begin{align}
T_{n,r}^{\beta}(x) = \sum_{k=1}^{\infty} P_{r}(k-1; \beta) \, L_{n,k}^{(\beta)}(x). \label{e10}  
\end{align}
The first few are:
\begin{align}
T_{n,0}^{\beta}(x) &= 1, \hspace{10mm}  T_{n,1}^{\beta}(x)=x + \frac{\beta (2-\beta)}{n(1-\beta)}(1-e^{-nx}), \nonumber\\
T_{n,2}^{\beta}(x) &= x^{2} +  \frac{2(1 + 2\beta - \beta^2)x}{n(1-\beta)} + \frac{\beta^{2} (3-\beta) }{n^2
(1-\beta)}(1-e^{-nx}) \label{e11} \\   
T_{n,3}^{\beta}(x) &= x^{3} + \frac{3(2+2\beta - \beta^2) \, x^2}{n \, (1-\beta)} + \frac{3 (2 + 4\beta + \beta^2
- 4\beta^3 + \beta^4) \, x}{n^2 \, (1-\beta)^2} + \frac{\beta^3 (4-\beta)}{n^3 \, (1-\beta)} \, (1 - e^{-n x}). \nonumber
\end{align}
\end{lemma}

\begin{proof}
Obviously by (\ref{e2}), we have  $T_{n,0}^{\beta}(x)=1.$  Next by definition of $T_{n,r}^{\beta}(x)$, we have
\begin{align*}
T_{n,r}^{\beta}(x) = \sum_{k=1}^{\infty} \frac{< L_{n,k-1}^{(\beta)}(t), t^{r} >}{< L_{n,k-1}^{(\beta)}(t), 1 >} \,
L_{n,k}^{(\beta)}(x)
= \sum_{k=1}^{\infty} P_{r}(k-1; \beta) \, L_{n,k}^{(\beta)}(x).
\end{align*}
Using Lemma \ref{l1} and Lemma \ref{l2}, we have
\begin{align*}
T_{n,1}^{\beta}(x) &= \sum_{k=1}^{\infty} P_{1}(k-1; \beta) \, L_{n,k}^{(\beta)}(x)
= \sum_{k=1}^{\infty} L_{n,k}^{(\beta)}(x) \, \frac{1}{n} \left[ (1-\beta) k + \frac{\beta (2-\beta)}{1-\beta} \right] \\
&= (1-\beta)B_n^{\beta}(t,x) + \frac{\beta(2-\beta)}{n(1-\beta)} \, (B_n^{\beta}(1,x)-e^{-nx}) \\
&= x + \frac{\beta (2-\beta)}{n (1-\beta)} (1-e^{-nx}) .
\end{align*}
\begin{align*}
T_{n,2}^{\beta}(x) &= \sum_{k=1}^{\infty} P_{2}(k-1; \beta) \, L_{n,k}^{(\beta)}(x) \\
&= \sum_{k=1}^{\infty} \frac{1}{n^{2}} \left[ (1-\beta)^{2} \, k^2 + (1 + 4\beta - 2\beta^2)
\, k + \frac{\beta^{2} (3-\beta)}{1-\beta} \right] \, L_{n,k}^{(\beta)}(x) \\
&= (1-\beta)^2 \, B_{n}^{\beta}(t^2,x) + \frac{1 + 4\beta - 2\beta^2}{n} \, B_{n}^{\beta}(t,x) + \frac{\beta^{2}
(3-\beta)}{n^2(1-\beta)} \, (B_{n}^{\beta}(1,x)-e^{-nx}) \\
&= x^{2} + \frac{2(1 + 2\beta - \beta^2)x}{n(1-\beta)} + \frac{\beta^{2} (3-\beta)}{n^2(1-\beta)}(1-e^{-nx}).
\end{align*}
A continuation of this process will provide $T_{n,r}^{\beta}(x)$ for cases of $r \geq 3$.
\end{proof}

\begin{lemma}\label{l5}
Define the polynomials
\begin{align*}
f_{n,r}^{\beta}(x) = T_{n,r}^{\beta}(x) + \left[ \left(\frac{\beta}{n}\right)^{r} \, \frac{r+1-\beta}{1-\beta}
- \delta_{r,0} \right] \, e^{- n x}
\end{align*}
then
\begin{align}\label{e12}  
f_{n,r}^{\beta}(x) = \sum_{j=0}^{r-1} \binom{r}{j} \, \frac{b_{j}^{r} \, x^{r-j}}{(n (1-\beta))^{j}} +
\left(\frac{\beta}{n}\right)^{r} \, \frac{r+1-\beta}{1-\beta}.
\end{align}
where the first few coefficients are given by
\begin{align*}
b_{0}^{r} &= 1 \hspace{15mm} b_{1}^{r} = r-1 + 2\beta - \beta^{2} \nonumber\\
b_{2}^{r} &= (r-1)(r-2) + 4 (r-2) \beta + (7-2r) \beta^{2} - 4 \beta^{3} + \beta^{4} \nonumber\\
b_{3}^{4} &= 6 + 12\beta + 6\beta^{2} - 8\beta^{3} - 6\beta^{4} + 6\beta^{5} - \beta^{6}  \\
b_{3}^{5} &= 24 + 36\beta + 6\beta^{2} - 20 \beta^{3} - 3\beta^{4} + 6\beta^{5} - \beta^{6} \nonumber\\
b_{4}^{5} &= 24 + 48\beta + 48\beta^{2} - 8\beta^{3} - 31 \beta^{4} + 8\beta^{5} + 14 \beta^{6} - 8 \beta^{7}
+ \beta^{8} \nonumber
\end{align*}
\end{lemma}

\begin{proof}
In view of (\ref{e5}) the general form of $P_{r}(k-1; \beta)$ is
\begin{align*}
P_{r}(k-1; \beta) = \frac{1}{n^{r}} \, \left[ (1-\beta)^{r} k^{r} + \sum_{s=1}^{r-1} \gamma_{s}^{r} \,
a_{s}^{r} \, k^{r-s-1} + \frac{\beta^{r} (1 + r - \beta)}{n^{r} \, (1-\beta)} \right],
\end{align*}
 where $\gamma_{s}^{r}$ are numeric coefficients, which leads to, for $r \geq 1$,
\begin{align*}
f_{n,r}^{\beta}(x) &= \sum_{k=1}^{\infty} P_{r}(k-1; \beta) \, L_{n,k}^{(\beta)}(x) + \frac{\beta^{r}
(1 + r - \beta)}{n^{r} \, (1-\beta)} \, e^{-n x} \\
&= (1-\beta)^{r} \, B_{n}^{\beta}(t^{r},x) + \sum_{s=1}^{r-1} \frac{\gamma_{s}^{r} \, a_{s}^{r}}{n^{s+1}} \,
B_{n}^{\beta}(t^{r-s-1}, x) + \frac{\beta^{r} (1 + r - \beta)}{n^{r} \, (1-\beta)} \, (B_{n}^{\beta}(1,x) \nonumber\\
& \hspace{10mm} - e^{-n x}) + \frac{\beta^{r} (1 + r - \beta)}{n^{r} \, (1-\beta)} \, e^{-n x} \\
&= (1-\beta)^{r} \, B_{n}^{\beta}(t^{r},x) + \sum_{s=1}^{r-1} \frac{\gamma_{s}^{r} \, a_{s}^{r}}{n^{s+1}} \,
B_{n}^{\beta}(t^{r-s-1}, x) + \frac{\beta^{r} (1 + r - \beta)}{n^{r} \, (1-\beta)}.
\end{align*}
Making use of (\ref{e2}), (\ref{e5}), and (\ref{e11}) the desired relation is obtained.
\end{proof}

\begin{lemma}\label{l6}
For $r \geq 2$ the polynomials $f_{n,r}^{\beta}(x)$ satisfy the relation
\begin{align}\label{e13}  
f_{n,r}^{\beta}(x) = \left[ x + \frac{2(r-1) + \beta(2-\beta)}{n \, (1-\beta)} \right] \, f_{n,r-1}^{\beta}(x)
+ \sum_{j=1}^{r-1} \frac{(-1)^{j} \, \beta^{j-1} \, \alpha_{j}^{r} }{(n (1-\beta))^{j+1}} \, f_{n,r-j-1}^{\beta}(x),
\end{align}
where the first few coefficients are given by
\begin{align*}
\alpha_{1}^{r} &= (r-1) (r-2 + 4\beta - \beta^2) \\
\alpha_{2}^{r} &= (r-2)(r-3) (2 + (2r-5) \beta + \beta^2) \\
\alpha_{3}^{4} &= 6 (1 + 6\beta - \beta^2)  \\
\alpha_{3}^{5} &= 12 (3 + 10\beta - 2\beta^{2})  \hspace{10mm} \alpha_{4}^{5} = 48 ( 1 + \beta + \beta^2)
\end{align*}
\end{lemma}

\begin{proof}
By considering
\begin{align*}
\lambda_{n,r}^{\beta}(x) = f_{n,r}^{\beta}(x) - \left[ x + \frac{2(r-1) + \beta(2-\beta)}{n \, (1-\beta)} \right]
\, f_{n,r-1}^{\beta}(x)
\end{align*}
it follows that the remaining polynomials can be expressed in terms of $\{ f_{n,0}^{\beta}(x), \cdots,
f_{n,r-2}^{\beta}(x) \}$.
\end{proof}



\section{Special Identities}

\begin{lemma}\label{l7}
The polynomials $T_{n,r}^{\beta}(x)$ satisfy the differential identity
\begin{align}\label{e14}  
& [-\beta \, x \, (D+n) + n \, x + \beta] \, [ n^{2} \, T_{n,r+2}^{\beta}(x) - n (r+\beta + 1) \, T_{n,r+1}^{\beta}(x)
+ \beta (r+2) \, T_{n,r}^{\beta}(x)] \nonumber\\
& \hspace{10mm} = n \, x^{2} \, (D + n) \, [ n(1-\beta) \, T_{n,r+1}^{\beta}(x) + \beta
(r+2) \, T_{n,r}^{\beta}(x)],
\end{align}
where $D = \frac{d}{dx}$.
\end{lemma}

\begin{proof}
By making use of (\ref{e9}) then it is seen that
\begin{align}\label{e15} 
& \sum_{k=1}^{\infty} k \, [ n(1-\beta) \, P_{r+1}(k-1; \beta) + \beta (r+2) \, P_{r}(k-1; \beta) ]
\, L_{n,k}^{(\beta)}(x) \nonumber\\
& \hspace{5mm} = n^{2} \, T_{n,r+2}^{\beta}(x) - n (r+\beta+1) \, T_{n,r+1}^{\beta}(x) + \beta (r+2) \,
T_{n,r}^{\beta}(x).
\end{align}
The basis functions $L_{n,k}^{(\beta)}(x)$ are seen to satisfy the differential identity
\begin{align}\label{e16} 
n \, x \, (D+n) \, L_{n,k}^{(\beta)}(x) = k \, \left[ - \beta \, (D+n) + \frac{n \, x + \beta}{x} \right]
\, L_{n,k}^{(\beta)}(x),
\end{align}
which, when combined with (\ref{e15}), leads to the desired result.
\end{proof}

\begin{lemma}\label{l8}
A differential recurrence relation between the modified Phillips operators, $T_{n,r}^{\beta}(x)$, and the
generalized Sz\'{a}sz-Mirakyan-Durrmeyer operators, $V_{n,r}^{\beta}(x)$, see \cite{vggr}, is given by
\begin{align}
& n \, \beta (1-\beta) \, \frac{d}{d\beta} \, T_{n,r}^{\beta}(x) + (r+1) \, \beta^{2} \, \frac{d}{d\beta} \,
T_{n,r-1}^{\beta}(x) \nonumber\\
& \hspace{2mm} = \beta \, n \, \left[ n \, V_{n,r+1}^{\beta}\left(x + \frac{\beta}{n} \right) - (r+1) \, V_{n,r}^{\beta}
\left( x + \frac{\beta}{n} \right) \right]  \label{e17} \\  
& \hspace{5mm} - \beta \, n^{2} \, T_{n,r+1}^{\beta}(x) - n (2r + 3\beta - \beta^{2}) \, T_{n,r}^{\beta}(x)
+ \beta (r+1) (r+\beta+1) \, T_{n,r-1}^{\beta}(x). \nonumber
\end{align}
\end{lemma}

\begin{proof}
The polynomials $P_{r}(k-1; \beta)$ satisfy
\begin{align}\label{e18} 
\beta \, \frac{d}{d\beta} \, P_{r}(k-1; \beta) = [r + \beta + (1-\beta) \, k] \, P_{r}(k-1;\beta) - n \,
P_{r+1}(k-1;\beta),
\end{align}
and the basis functions $L_{n,k}^{(\beta)}(x)$ are seen to satisfy the differential identity
\begin{align}\label{e19} 
\frac{d}{d \beta} \, L_{n,k}^{(\beta)}(x) = - k \, L_{n,k}^{(\beta)}(x) + (k-1) \, L_{n,k-1}^{(\beta)}\left( x
+ \frac{\beta}{n} \right).
\end{align}
Differentiating (\ref{e10}) with respect to $\beta$ leads to
\begin{align*}
\beta \, \frac{d}{d\beta} \, T_{n,r}^{\beta}(x) &= \beta \, \sum_{k=1}^{\infty} \left[ P_{r}(k-1;\beta) \,
\left(\frac{d}{d\beta} \, L_{n,k}^{(\beta)}(x) \right) + \left( \frac{d}{d\beta} \, P_{r}(k-1; \beta) \right)
\, L_{n,k}^{(\beta)}(x) \right]
\end{align*}
and can be seen in the form
\begin{align*}
& \left( \beta \, \partial_{\beta} - r - \beta \right) \, T_{n,r}^{\beta}(x) + n \, T_{n,r+1}^{\beta}(x) \\
& \hspace{2mm} = \beta \, \sum_{k=0}^{\infty} k \, P_{r}(k;\beta) \, L_{n,k}^{(\beta)}\left(x + \frac{\beta}{n}\right)
+ (1-2\beta) \, \sum_{k=1}^{\infty} k \, P_{r}(k-1;\beta) \, L_{n,k}^{(\beta)}(x).
\end{align*}
Now, by using (\ref{e15}) and the identity given in the proof of Lemma 5 of \cite{vggr}, then the desired
identity is obtained.
\end{proof}

\begin{remark} \label{r1}
If we denote the central moment as $\mu_{n,r}^\beta(x)=P_{n}^{\beta}((t-x)^r,x)$, then
\begin{align}
\mu_{n,1}^{\beta}(x) &=  F_{n,1}^{\beta}(x) \, (1-e^{-nx}), \nonumber\\
\mu_{n,2}^{\beta}(x) &=  \frac{2 (1+ 2\beta - \beta^2) \, x}{n(1-\beta)} +  F_{n,2}^{\beta}(x) \, (1-e^{-nx})
\label{e20} \\ 
\mu_{n,3}^{\beta}(x) &= \frac{3 \, \beta (\beta - 2) \, x^2}{n \, (1-\beta)} + 3 \, (2 + 4\beta + \beta^2 -
\beta^3 + \beta^4) \, \frac{x}{n^2 \, (1-\beta)^2} + F_{n,3}^{\beta}(x) \, (1 - e^{-nx}) \nonumber\\
\mu_{n,4}^{\beta}(x) &= \frac{4 \, \beta (2-\beta) \, x^3}{n \, (1-\beta)} + 2 \, (10 + 8\beta -13 \beta^2
+ 6\beta^3 + 3\beta^4) \, \frac{x^2}{n^2 \, (1-\beta)^2} \nonumber\\
& \hspace{10mm} + 4 \, (6 + 12\beta + 6 \beta^2 - 8\beta^3 - 6\beta^4 + 6\beta^5 - \beta^6) \, \frac{x}{n^3 \,
(1-\beta)^3} + F_{n,4}^{\beta}(x) \, (1 - e^{-nx}) \nonumber\\
\mu_{n,5}^{\beta}(x) &= \frac{5 \, \beta (\beta -2) \, x^4}{n \, (1-\beta)} + \frac{10 \, \beta^2 (3 - 4\beta
+ \beta^2) \, x^3}{n^2 \, (1-\beta)^2} + \frac{10 \, \lambda_{3}^{5} \, x^2}{n^3 \, (1-\beta)^3}
\nonumber\\
& \hspace{10mm}  + \frac{5 \, \lambda_{4}^{5} \, x}{n^4 \, (1-\beta)^4} + F_{n,5}^{\beta}(x) \, (1 - e^{-n x})
\nonumber
\end{align}
where
\begin{align}\label{e21} 
F_{n,r}^{\beta}(x) = \sum_{s=0}^{r-1} (-1)^{s} \binom{r}{s} \, \left(\frac{\beta}{n} \right)^{r-s} \,
\frac{r+1-s-\beta}{1-\beta} \, x^{s}
\end{align}
and
\begin{align*}
\lambda_{3}^{5} &= 12 + 12 \beta - 6 \beta^2 - 4 \beta^3 + 9 \beta^4 - 6 \beta^5 + \beta^6 \\
\lambda_{4}^{5} &= 23 + 38 \beta + 27 \beta^2 - 12 \beta^3 - 25 \beta^4 + 8 \beta^5 + 14 \beta^6 - 8 \beta^7 + \beta^8.
\end{align*}
\end{remark}



\section{Direct Estimates}

In this section, we establish the following direct result:

\begin{remark}\label{r2}
Let $f$ be a continuous function on $[0,\infty)$ for $n\to \infty$, the sequence $\{P_{n}^{\beta}(f,x)\}$ converges
uniformly to $f(x)$ in $[a,b]\subset[0,\infty)$, which follows from the well known Boham-Korovkin theorem.
\end{remark}

\begin{theorem}\label{t2} Let $f$ be a bounded integrable function on $[0,\infty)$ and has second derivative at a point
$x\in [0,\infty)$, then $$\lim_{n\to \infty} n[P_{n}^{\beta}(f,x)-f(x)] = \frac{\beta(2-\beta)}{1-\beta} f^{\prime}(x)
+ \frac{(1+2\beta-\beta^2)x}{1-\beta} \, f^{\prime\prime}(x).$$
\end{theorem}
\begin{proof}
By the Taylor's expansion of $f$, we have
\begin{equation}\label{e22}  
f(t) = f(x) + f^{\prime}(x)(t-x) + \frac{1}{2} f^{\prime\prime}(x)(t-x)^{2} + r(t,x)(t-x)^{2},
\end{equation}
where $r(t,x)$ is the remainder term  and $\displaystyle\lim_{n\rightarrow\infty}r(t,x)=0.$
Operating $P_{n}^{\beta}$ to the equation (\ref{e22}), we obtain
\begin{align*}
P_{n}^{\beta}(f,x) - f(x) &= P_{n}^{\beta}(t-x,x) f^{\prime}(x) + P_{n}^{\beta}\left(\left(t-x\right)^{2},x\right)
\frac{f^{\prime\prime}(x)}{2} \\
& \hspace{10mm} + P_{n}^{\beta}\left( r\left( t,x\right)  \left(  t-x\right)^{2},x\right).
\end{align*}
Using the Cauchy-Schwarz inequality, we have
\begin{equation}\label{e23} 
P_{n}^{\beta}\left(r\left(t,x\right)\left(t-x\right)^{2},x\right)\leq \sqrt{P_{n}^{\beta}\left(r^{2}\left(t,x\right)
,x\right)} \, \sqrt{P_{n}^{\beta}\left(\left(t-x\right)^{4},x\right)}.
\end{equation}
As $r^{2}\left(x,x\right)=0$ and $r^{2}\left(t,x\right) \in C_{2}^{\ast}[0,\infty)$, we have
\begin{equation}\label{e24} 
\lim_{n\rightarrow\infty}P_{n}^{\beta}\left(r^{2}\left(t,x\right),x\right)=r^{2}\left(x,x\right) = 0
\end{equation}
uniformly with respect to $x\in\left[0,A\right].$ Now from (\ref{e23}), (\ref{e24}) and from Remark \ref{r1}, we get
\begin{align*}
\lim_{n\rightarrow\infty} n P_{n}^{\beta}\left(r\left(t,x\right) \left(t-x\right)^{2}, x\right) = 0.
\end{align*}
Thus
\begin{align*}
\lim_{n\rightarrow\infty} n \left(P_{n}^{\beta}(f,x)-f(x)\right)
&= \lim_{n\rightarrow\infty} n \biggl[P_{n}^{\beta}(t-x,x) f^{\prime}(x) + \frac{1}{2}f^{\prime\prime
}(x) P_{n}^{\beta}(\left(t-x\right)^{2},x) \\
& \hspace{25mm} + P_{n}^{\beta}(r\left(t,x\right)\left(t-x\right)^{2},x)\biggr] \\
&= \frac{\beta(2-\beta)}{1-\beta} f^{\prime}(x) + \frac{(1+2\beta-\beta^2)x}{1-\beta}f^{\prime\prime}(x).
\end{align*}
\end{proof}

By $C_B[0,\infty)$, we denote the class on real valued continuous bounded functions $f(x)$ for $x\in[0,\infty)$ with
the norm $||f||=\sup_{x\in[0,\infty)}|f(x)|$. For $f\in C_B[0,\infty)$ and $\delta >0$ the $m$-th order modulus
of continuity is defined as
$$\omega_m(f,\delta)=\sup_{0\le h\le \delta}\sup_{x\in[0,\infty)}|\Delta^m_hf(x)|,$$
where $\Delta$ is the forward difference. For the case $m=1$, we mean the usual modulus of continuity denoted by
$\omega(f,\delta).$
The Peetre's $K$-functional is defined as
$$K_2(f,\delta) = \inf_{g\in C_B^2[0,\infty)} \left\{ ||f-g|| + \delta ||g^{\prime\prime}||:g\in C_B^2[0,\infty)\right\},$$
where $$C_B^2[0,\infty)=\{g \in C_B[0,\infty): g^\prime, g^{\prime\prime} \in C_B[0,\infty)\}.$$

\begin{theorem}\label{t3} Let $f\in C_B[0,\infty)$ and $\beta>0,$ then
\begin{align*}
|P_{n}^{\beta}(f,x)-f(x)| &\leq C \omega_2\left(f,\sqrt{\delta_n}\right) + \omega\bigg(f,\frac{\beta (2-\beta)}
{n (1-\beta)} (1-e^{-nx})\bigg),
\end{align*}
where
$$\delta_{n} = \frac{ 2 (1+2\beta-\beta^2) x }{ n(1-\beta) } + (1-e^{-nx}) \, \left[\frac{\beta^2 (3-\beta)}
{n^2 (1-\beta)} - \frac{2\beta x(2-\beta)}{n(1-\beta)}\right] + \bigg(\frac{\beta (2-\beta)}{n (1-\beta)} (1-e^{-nx})\bigg)^2$$
and $C$ is a positive constant.
\end{theorem}

\begin{proof}
We introduce the auxiliary operators $\bar{P}_{n}^{\beta}:C_B[0,\infty)\to C_B[0,\infty)$ as follows
\begin{align}\label{e25} 
\bar{P}_{n}^{\beta}(f,x) = P_{n}^{\beta}(f,x)-f\bigg(x + \frac{\beta (2-\beta)}{n (1-\beta)} (1-e^{-nx})\bigg)+f(x),
\end{align}
These operators are linear and preserves the linear functions in view of Lemma \ref{l4}. Let $g\in C_B^2[0,\infty)$
and $x, \, t \in[0,\infty).$ By Taylor's expansion
\begin{align*}
g(t) = g(x) + (t-x) \, g^\prime(x) + \int_{x}^{t}(t-u) \, g^{\prime\prime}(u) \, du,
\end{align*}
we have
\begin{align}
|\bar{P}_{n}^{\beta}(g,x)-g(x)| &\leq \bar{P}_{n}^{\beta}\bigg(\bigg|\int_{x}^{t}(t-u)g^{\prime\prime}(u)du\bigg|,x\bigg)
\nonumber\\
&\leq P_{n}^{\beta}\bigg(\bigg|\int_{x}^{t}(t-u)g^{\prime\prime}(u)du\bigg|,x\bigg) \nonumber\\
& \hspace{20mm} + \bigg|\int_{x}^{x + \frac{\beta (2-\beta)}{n (1-\beta)} (1-e^{-nx})}\bigg(x + \frac{\beta (2-\beta)}
{n (1-\beta)} (1-e^{-nx})-u\bigg)g^{\prime\prime}(u)du\bigg|
\nonumber\\
&\leq P_{n}^{\beta}((t-x)^2,x)\|g^{\prime\prime}\|+\bigg|\int_{x}^{x + \frac{\beta (2-\beta)}{n (1-\beta)} (1-e^{-nx})}
\bigg(\frac{\beta (2-\beta)}{n (1-\beta)} (1-e^{-nx})\bigg)du\bigg|\|g^{\prime\prime}\|\nonumber
\end{align}
Next, using Remark \ref{r1}, we have
\begin{align}
|\bar{P}_{n}^{\beta}(g,x)-g(x)| &\leq \bigg[ P_{n}^{\beta}((t-x)^2,x)+\bigg(\frac{\beta (2-\beta)}{n (1-\beta)}
(1-e^{-nx})\bigg)^2\bigg]\|
g^{\prime\prime}\|\nonumber\\
&=\delta_n\|
g^{\prime\prime}\|,\label{e26} 
\end{align}
where $\delta_n=\mu_{n,2}^{\beta}(x)+\bigg(\frac{\beta (2-\beta)}{n (1-\beta)} (1-e^{-nx})\bigg)^2.$

Since
\begin{align*}
|P_{n}^{\beta}(f,x)|\le \sum_{k=1}^{\infty} \left( \int_{0}^{\infty} L_{n,k-1}^{(\beta)}(t) \, dt \right)^{-1}
L_{n,k}^{(\beta)}(x) \, \int_{0}^{\infty} L_{n,k-1}^{(\beta)}(t) \, |f(t)| \, dt + e^{-nx} \, |f(0)| \le \|f\|,
\end{align*}
then by use of (\ref{e23}) we have
\begin{align}\label{e27} 
||\bar{P}_{n}^{\beta}(f,x)|| \le ||P_{n}^{\beta}(f,x)|| + 2||f|| \le 3||f||, \hspace{5mm} f \in C_B[0,\infty).
\end{align}
Using (\ref{e25}), (\ref{e26}) and (\ref{e27}), we have
\begin{align*}
|P_{n}^{\beta}(f,x)-f(x)| &\leq  |\bar{P}_{n}^{\beta}(f-g,x)-(f-g)(x)|+|\bar{P}_{n}^{\beta}(g,x)-g(x)|\\
& \hspace{25mm} + \bigg|f\bigg(x + \frac{\beta (2-\beta)}{n (1-\beta)} (1-e^{-nx})\bigg)-f(x)\bigg|\\
&\leq 4\|f-g\|+\delta_n\|g^{\prime\prime}\|+\bigg|f\bigg(x + \frac{\beta (2-\beta)}{n (1-\beta)} (1-e^{-nx})\bigg)-f(x)\bigg|\\
&\leq C\left\{\|f-g\|+\delta_n\|g^{\prime\prime}\|\right\}+\omega\left(f,\frac{\beta (2-\beta)}{n (1-\beta)} (1-e^{-nx})\right).
\end{align*}
Taking infimum over all $g\in C_B^2[0,\infty)$, and using the inequality $K_2(f,\delta) \leq
C\omega_2(f,\sqrt{\delta})$, $\delta >0$ due to \cite{dl}, we get the desired assertion.
\end{proof}



\end{document}